\documentclass{article}
\usepackage{amsmath,amssymb,theorem}
\usepackage{verbatim}

\usepackage{color}
\usepackage{makeidx}

\theorembodyfont{\upshape}

 \makeatletter
 \@addtoreset{equation}{section}
 \makeatother

 \newtheorem{ittheorem}{Theorem}
 \newtheorem{itlemma}{Lemma}
 \newtheorem{itproposition}{Proposition}
 \newtheorem{itdefinition}{Definition}

 \newtheorem{itremark}{Remark}
 \newtheorem{itclaim}{Claim}
 \newtheorem{itcorollary}{\bf Corollary}

 \newenvironment{theorem}{\addtocounter{equation}{1}
 \begin{ittheorem}}{\end{ittheorem}}

 \newenvironment{lemma}{\addtocounter{equation}{1}
 \begin{itlemma}}{\end{itlemma}}

 \newenvironment{proposition}{\addtocounter{equation}{1}
 \begin{itproposition}}{\end{itproposition}}

 \newenvironment{definition}{\addtocounter{equation}{1}
 \begin{itdefinition}}{\end{itdefinition}}

 \newenvironment{remark}{\addtocounter{equation}{1}
 \begin{itremark}}{\end{itremark}}

 \newenvironment{claim}{\addtocounter{equation}{1}
 \begin{itclaim}}{\end{itclaim}}

 \newenvironment{proof}{\noindent {\bf Proof.\,}
 }{\hspace*{\fill}$\qed$\medskip}

 \newenvironment{corollary}{\addtocounter{equation}{1}
 \begin{itcorollary}}{\end{itcorollary}}

 \newcommand{\be}[1]{\begin{eqnarray*}\label{#1}}
 \newcommand{\ee}{\end{eqnarray*}}

 \newcommand{\bl}[1]{\begin{lemma}\label{#1}}
 \newcommand{\el}{\end{lemma}}

 \newcommand{\br}[1]{\begin{remark}\label{#1}}
 \newcommand{\er}{\end{remark}}

 \newcommand{\bt}[1]{\begin{theorem}\label{#1}}
 \newcommand{\et}{\end{theorem}}

 \newcommand{\bd}[1]{\begin{definition}\label{#1}}
 \newcommand{\ed}{\end{definition}}

 \newcommand{\bcl}[1]{\begin{claim}\label{#1}}
 \newcommand{\ecl}{\end{claim}}

 \newcommand{\bp}[1]{\begin{proposition}\label{#1}}
 \newcommand{\ep}{\end{proposition}}

 \newcommand{\bc}[1]{\begin{corollary}\label{#1}}
 \newcommand{\ec}{\end{corollary}}

 \newcommand{\bpr}{\begin{proof}}
 \newcommand{\epr}{\end{proof}}

 \newcommand{\bi}{\begin{itemize}}
 \newcommand{\ei}{\end{itemize}}

 \newcommand{\ben}{\begin{enumerate}}
 \newcommand{\een}{\end{enumerate}}

\def\un{{1\cdots 1}}

\def\un{\{\,1,\dots,N\,\}}

\def\uro{\smash{{U}^{\!\!\!\!\raise5pt\hbox{$\scriptstyle o$}}}}

\def\bp{{\overline{p}}}

\def\bp{{\overline{p}}}

 \def \ba {\begin{array}}
 \def \ea {\end{array}}

 \def \qed {{\heartsuit\hfill}}
 
 \def \R {{\mathbb R}}
 
 \def \N {{\mathbb N}}

 \def \cS {{\cal S}}

\def \qed {{\square\hfill}}

  \def\cS{{\cal S}}

\def \qed {{\square\hfill}}

\def\R{{\mathbb R}}

\def\N{{\mathbb N}}

\def\eqref#1{(\ref{#1})}

\makeindex

\begin{document}

\title{A Markov chain representation of the 
Perron--Frobenius eigenvector}

 \author{
Rapha\"el Cerf and Joseba Dalmau
\\
DMA, 
{\'E}cole Normale Sup\'erieure\\
}

\maketitle



\begin{abstract}
\noindent
We consider the problem of finding the Perron--Frobenius
eigenvector of a primitive matrix. Dividing each of the 
rows of the matrix by the sum of the elements in the row,
the resulting new matrix is stochastic.
We give a formula for the Perron--Frobenius eigenvector
of the original matrix, in terms of a realization
of the Markov chain defined by the associated stochastic matrix.
This formula is a generalization of the classical formula
for the invariant probability measure of a Markov chain.
\noindent
\end{abstract}


\allowdisplaybreaks[4]

\noindent
Let $A$ be a primitive matrix of size $N$,
i.e., a non--negative matrix whose $m$--th 
power is positive for some natural number $m$.
The Perron--Frobenius theorem (theorem 1.1 in~\cite{SEN})
states that there 
exist a positive real number $\lambda$ and a vector $u$
on the unit simplex $\lbrace\,x\in\R_+^N:x_1+\cdots+x_N=1\,\rbrace$ such that 
$u^TA=\lambda u^T$.
Moreover, the eigenvalue $\lambda$ is simple,
is larger in absolute value than any other eigenvalue of $A$,
and any non--negative eigenvector of $A$ is a multiple of $u$.
The eigenvalue $\lambda$ is the Perron--Frobenius eigenvalue of $A$ and $u$
is the Perron--Frobenius eigenvector of $A$.
The purpose of this note is to give a Markov chain representation
of the Perron--Frobenius eigenvector $u$.

\noindent
The matrix $A$ can be decomposed as $A(i,j)=f(i)M(i,j)$
with $f(i)$ being the sum of the elements in the $i$--th row of $A$
and $M(i,j)=A(i,j)/f(i)$.
The matrix $M$ is now primitive and stochastic,
so that it naturally defines an ergodic Markov chain.
Let $(X_n)_{n\in\N}$ be a Markov chain with state space
$\lbrace\,1,\dots,N\,\rbrace$ and transition matrix $M$,
denote by $E_k$ the expectation of the Markov chain issued from $k$
and $\tau_k$ the time of the first return of the chain to $k$.
We have the following result.
\begin{theorem}
Let $1\leq k\leq N$. 
The Perron--Frobenius eigenvector $u$ of $A$ is given by the formula
$$\forall i\in\un\qquad
u_i\,=\,\frac{
\displaystyle
E_k\Bigg(\sum_{n=0}^{\tau_k-1}\Big(
1_{\{X_n=i\}}
\lambda^{-n}\prod_{t=0}^{n-1}f(X_t)
\Big)
\Bigg)
}
{
\displaystyle
E_k\Bigg(\sum_{n=0}^{\tau_k-1}\Big(
\lambda^{-n}\prod_{t=0}^{n-1}f(X_t)
\Big)
\Bigg)
}\,.
$$
\end{theorem}
By taking $i=k$ in the above formula we obtain the following corollary.
\begin{corollary}
The Perron--Frobenius eigenvector $u$ of $A$
is given by the formula
$$\forall k\in\un\qquad
u_k\,=\,\frac{1}{
\displaystyle
E_k\Bigg(\sum_{n=0}^{\tau_k-1}\Big(
\lambda^{-n}\prod_{t=0}^{n-1}f(X_t)
\Big)
\Bigg)
}\,.
$$
\end{corollary}
This formula is a generalization of the classical formula for
the invariant probability measure of a Markov chain.
Indeed, in the particular case where $A$ is stochastic,
$f$ is constant equal to 1, $\lambda$ is also equal to 1,
and $u$ corresponds to the invariant probability measure of the 
Markov chain. Thus, the formula of the corollary becomes 
the well--known formula 
$$\forall k\in\un\qquad
u_k\,=\,\frac{1}{
\displaystyle
E_k(\tau_k)
}\,.$$
Before proving the theorem, we state a preparatory lemma.
\begin{lemma}
\label{cytra}
Let $A$ be a non--negative primitive matrix 
of size $N$.
Its Perron--Frobenius eigenvalue~$\lambda$
satisfies the following identity: for any
$k\in\un$,
$$\displaylines{
1\,=\,\frac{1}{\lambda}A(k,k)+
\frac{1}{\lambda^2}\sum_{i_1\neq k}A(k,i_1)A(i_1,w)+\cdots\hfil\cr
\hfil+\frac{1}{\lambda^{n}}\sum_{i_1,\dots,i_{n-1}\neq k}A(k,i_1)A(i_1,i_2)\cdots
A(i_{n-1},k)+\cdots}$$
\end{lemma}
\begin{proof}
Let $(x_i)_{1\leq i\leq N}$ be a non--negative eigenvector associated to the Perron--Frobenius eigenvalue~$\lambda$
of $A$:
$$\forall j\in\un\qquad\sum_{i=1}^N x_i A(i,j)\,=\,\lambda x_j\,.$$
Since $A$ is primitive, all the components of $x$ are positive.
Let $1\leq k\leq N$ be fixed.
We have thus
$$1\,=\,\frac{1}{\lambda x_k}
\sum_{i=1}^N x_i A(i,k)\,=\,
\frac{1}{\lambda}A(k,k)+
\sum_{i\neq k} 
\frac{x_i}{\lambda x_k}
A(i,k)\,.$$
We replace $x_i$ in the last sum and we get
\begin{align*}
1\,&=\,
\frac{1}{\lambda}A(k,k)+
\sum_{i\neq k} \sum_{i'=1}^N
\frac{x_{i'}}{\lambda^2 x_k}
A(i',i)A(i,k)\cr
&=\,
\frac{1}{\lambda}A(k,k)+
\sum_{i\neq k} 
\frac{1}{\lambda^2 }
A(k,i)A(i,k)\,+\,
\sum_{i,i'\neq k} 
\frac{x_{i'}}{\lambda^2 x_k}
A(i',i)A(i,k)\,.
\end{align*}
Iterating this procedure, we obtain, for $n\geq 1$,
$$\displaylines{
1\,=\,\sum_{t=0}^{n-1}
\frac{1}{\lambda^{t+1}}\sum_{i_1,\dots,i_{t}\neq k}A(k,i_1)A(i_1,i_2)\cdots
A(i_{t},k)\,+\hfil\cr
\hfil
\frac{1}{\lambda^{n}}
\sum_{i_1,\dots,i_{n-1}\neq k}
\frac{x_{i_1}}{x_k}
A(i_1,i_2)\cdots
A(i_{n-1},k)\,.}$$
Let $B$ be the matrix obtained from $A$ by filling with zeroes
the line and the column associated 
to~$k$ . The last term of the previous identity can be rewritten as
$$
\frac{1}{\lambda^{n}}
\sum_{i,j\neq k}
\frac{x(i)}{x(k)}
B(i,j)^{n-2}
A(j,k)\,.$$
Yet it follows from part~(e) of theorem~$1.1$ of \cite{SEN} that the spectral
radius of $B$ is strictly less than~$\lambda$, whence
$$\forall i,j\in \un\qquad
\lim_{n\to\infty}
\frac{1}{\lambda^{n}}
B(i,j)^{n-2}\,=\,0\,.$$
Thus the previous sum vanishes as $n$ goes to $\infty$. Passing to the limit,
we obtain the desired identity.
\end{proof}

\noindent
We now proceed to the proof of the theorem.

\medskip
\begin{proof}
Let us note $E_k$ and $\tau_k$ simply by $E$ and $\tau$.
We set, for $1\leq i\leq N$,
$$y_i\,=\,
\displaystyle
E\Bigg(\sum_{n=0}^{\tau-1}\Big(
1_{\{X_n=i\}}
\lambda^{-n}\prod_{t=0}^{n-1}f(X_t)
\Big)
\Bigg)
\,.$$
Obviously, the vector $(y_i)_{1\leq i\leq N}$ is non--null and its components are non--negative. 
Let us compute
\begin{align*}
\sum_{i=1}^N
y_i f(i)&M(i,j)\,=\,\cr
&\phantom{=}\,\sum_{i=1}^N\sum_{n\geq 0}
E\Bigg(1_{\{\tau>n\}}
\lambda^{-n}
\Big(
\prod_{t=0}^{n-1}f(X_t)
\Big)
1_{\{X_n=i\}}
f(i)M(i,j)
\Bigg)\cr
&=\,\sum_{i=1}^N\sum_{n\geq 0}
E\Bigg(1_{\{\tau>n\}}
\lambda^{-n}
\Big(
\prod_{t=0}^{n}f(X_t)
\Big)
1_{\{X_n=i\}}
1_{\{X_{n+1}=j\}}
\Bigg)\cr
&=\,
E\Bigg(\sum_{n=0}^{\tau-1}
1_{\{X_{n+1}=j\}}
\lambda^{-n}
\Big(
\prod_{t=0}^{n}f(X_t)
\Big)
\Bigg)\cr
&=\,
\lambda\,E\Bigg(\sum_{n=1}^{\tau}
1_{\{X_{n}=j\}}
\lambda^{-n}
\Big(
\prod_{t=0}^{n-1}f(X_t)
\Big)
\Bigg)\,.
\end{align*}
Suppose that $j\neq k$. Then the term in the last sum vanishes 
for $n=0$ or $n=\tau$, and we recover the identity
$$
\sum_{i=1}^N y_if(i)M(i,j)\,=\,\lambda y_j\,.$$
For $j=k$, we obtain
$$
\sum_{i=1}^N y_if(i)M(i,j)\,=\,
\lambda\,E\Bigg(
\lambda^{-\tau}
\prod_{t=0}^{\tau-1}f(X_t)
\Bigg)\,.
$$
The last expectation can be rewritten as
$$\displaylines{
E\Bigg(
\lambda^{-\tau}
\prod_{t=0}^{\tau-1}f(X_t)
\Bigg)
\,=\,
\sum_{n\geq 1}
E\Bigg(
1_{\{\tau=n\}}
\lambda^{-n}
\prod_{t=0}^{n-1}f(X_t)
\Bigg)\hfill\cr
\,=\,
\sum_{n\geq 1}
\sum_{i_1,\dots,i_{n-1}\neq k}
\lambda^{-n}f(k)f(i_1)\cdots f(i_{n-1})
\hfill\cr
\hfill\times
P\big(X_1=i_1,\dots,X_{n-1}=i_{n-1},X_n=k\big)\cr
\,=\,
\sum_{n\geq 1}
\sum_{i_1,\dots,i_{n-1}\neq k}
\lambda^{-n}f(k)M(k,i_1)\cdots f(i_{n-1})M(i_{n-1},k)
\,.
}$$
This last sum is equal to~$1$ by lemma~\ref{cytra}.
Noticing that
$y_k=1$, we conclude that
$$
\sum_{i=1}^N y_if(i)M(i,k)\,=\,\lambda y_k\,.$$
Therefore
the vector $(y_i)_{1\leq i\leq k}$ is an eigenvector of $A$ associated to~$\lambda$.
We normalize it so that it belongs to the unit simplex and we obtain
the formula stated in the theorem.
\end{proof}

\noindent
The decomposition of the matrix $A$
in terms of $f$ and $M$ might seem artificial.
However,
it arises naturally in several situations.
We illustrate this fact in the following two examples,
which also provided us with the motivation to construct the 
probabilistic representation of the Perron--Frobenius
eigenvector given in the above results.
\medskip

\noindent {\bf Mutation--selection equilibrium.}
Consider a mutation--selection model
in which individuals have an associated type,
the possible types being numbered from $1$ to $N$.
Individuals reproduce and mutate,
and mutations, which only happen during the reproduction events,
change the type of the offspring.
We fix a function $f:\un\rightarrow\,]0,+\infty[\,$
and a primitive stochastic matrix $(M(i,j))_{1\leq i,j\leq N}$.
An individual of type $i$ reproduces at rate $f(i)$,
and the offspring mutates to type $j$ with probability $M(i,j)$.
Let the vector $(x_k)_{1\leq k\leq N}$ represents
the proportions of the different types in a population at equilibrium.
Because of the equilibrium assumption,
the rate of creation of type $k$ individuals 
must be equal to the rate of destruction
of type $k$ individuals.
If the rate of destruction of an individual is independent 
of its type,
we obtain the mutation--selection equilibrium equation
$$\forall k\in \un\qquad
x_k\,\sum_{i=1}^N x_if(i)\,=\,
\,\sum_{i=1}^N x_if(i)M(i,k)\,.\qquad\qquad(\cal S)$$
This equilibrium equation is of high interest and it arises in a wide 
variety of models,
for instance in Eigen's quasispecies model~\cite{ECS1}.
The main question is whether a solution of $(\cS)$ exists
in the $N-1$ dimensional unit simplex, and whether the solution, if it exists,
is unique or not.
In view of the Perron--Frobenius theorem,
the unique solution of $(\cal S)$ in the unit simplex
is given by the Perron--Frobenius eigenvector $u$
of the matrix $A=(
f(i)M(i,j))_{1\leq i,j\leq N}$.
The Perron--Frobenius eigenvalue $\lambda$ 
corresponds to the mean fitness at equilibrium
$\lambda=\sum_{1\leq i\leq N} u_if(i)$.
In this particular setting,
the Markov chain $(X_n)_{n\geq0}$ can be naturally
interpreted as the random walk of a mutant in a neutrally
evolving population, that is, if $f$ is constant equal to 1.
\medskip

\noindent {\bf Multitype Galton--Watson process.} 
We consider next a probabilistic counterpart of the mutation--selection
equilibrium above. Consider a multitype Galton--Watson process
in which individuals of type $i$ produce offsprings according
to a law with mean $f(i)$ and finite variance,
and the offspring of a type $i$ individual
becomes of type $j$ with probability $M(i,j)$.
The matrix $A=(f(i)M(i,j))_{1\leq i,j\leq N}$
is known as the mean matrix of the process.
It is well--known (chapter 2 of~\cite{Harris}) that if the Perron--Frobenius eigenvalue $\lambda$
of $A$ is strictly larger than one, the multitype Galton--Watson process
has a positive probability of survival.
Conditioned on the survival event,
the vector of proportions of the different types 
converges to $u$ when time goes to $\infty$, $u$
being the Perron--Frobenius eigenvector of $A$. 

\bibliographystyle{plain}
\bibliography{reqs2}
 \thispagestyle{empty}

\end{document}